\documentclass[a4paper]{article}
\usepackage{amsmath,amssymb,amsfonts,amsthm,mathrsfs,color}

\usepackage{url}

\let\bs\boldsymbol
\let\cal\mathcal

\def\R{\mathbb{R}}

\def\HH{\mathbb{H}}

\newcommand{\qi}{\mathbf i}
\newcommand{\qj}{\mathbf j}
\newcommand{\qk}{\mathbf k}

\DeclareMathOperator{\re}{Re}
\DeclareMathOperator{\V}{Vec}

\newcommand{\doprodstar}[1]{\ooalign{\ooalign{$#1\prod$\cr\hfil$#1\ast$\hfil\cr}\cr\hfil$#1\phantom\prod$\hfil\cr}}
\newcommand{\prodstar}{\mathop{\mathrel{\mathpalette\doprodstar\relax}}}

\newtheorem{theorem}{Theorem}
\newtheorem{lemma}{Lemma}

\newtheorem{example}{{\bf Example}}
\newtheorem{remark}{Remark}

\begin{document}
\title{Weierstrass method for quaternionic polynomial root-finding}%

\author{M.~Irene Falc\~{a}o\\
CMAT and DMA\\ University of Minho, Portugal\\
mif@math.uminho.pt\\
\and Fernando Miranda\\
CMAT and DMA\\ University of Minho, Portugal\\
fmiranda@math.uminho.pt\\
\and Ricardo Severino\\
DMA\\ University of Minho, Portugal\\
ricardo@math.uminho.pt\\
\and M. Joana Soares\\
NIPE and DMA\\ University of Minho, Portugal\\
jsoares@math.uminho.pt
}

\maketitle

\begin{abstract}
Quaternions, introduced by Hamilton in 1843 as a generalization of complex numbers, have found, in more recent years, a wealth of applications in a number of different areas which motivated the design of efficient methods for numerically approximating the zeros of quaternionic polynomials. In fact, one can find in the literature recent contributions to this subject based on the use of complex techniques, but numerical methods relying  on quaternion arithmetic remain scarce. In this paper we propose a Weierstrass-like method for finding simultaneously {\sl all} the zeros of unilateral quaternionic  polynomials. The convergence analysis and several numerical examples illustrating the performance of the method are also presented.

\end{abstract}


\section{Introduction}
The increasing interest in using quaternions and their applications in areas as diverse as number theory, robotics, virtual reality or image processing (see e.g. \cite{BrackxHitzerSangwine2013,Farouki2017,PereiraRocha2008,PereiraRochaVettori2005,PereiraVettori2006}), motivated several authors to consider extending well-known (complex) numerical methods, in particular root-finding methods, to the quaternion algebra framework.
However, the problem of finding the zeros of quaternionic polynomials turns out to be much more demanding than the analogous problem over the real and complex fields. 
Niven, in his pioneering work, \cite{Niven1941}, gave a first extension of the Fundamental Theorem of Algebra for the quaternion context, proving that any quaternionic polynomial of positive degree whose coefficients are located only on one side of the powers must have at least one quaternionic root. In the aforementioned paper, Niven also proposed a method for computing the roots of such polynomials. This algorithm is, however, as stated by Niven a ``not very practical'' one, due to the need of solving two coupled nonlinear equations for the determination of pairs of real constants. Later, in \cite{SerodioPereiraVitoria2001}, the authors, by making use of (a complexified version of) the companion matrix of the polynomial, turned the ideas of Niven into what can be considered as the first really usable numerical algorithm. 
    
Nowadays, other quaternionic root-finding algorithms are available which essentially replace the problem of computing the roots of a quaternionic polynomial of degree $n$, by the problem of determining the roots of a real or complex polynomial of degree $2n$ (usually with multiple roots), relying in this way on algorithms for complex polynomial root-finding  (see \cite{DeLeoDucatiLeonardi2006,SerodioSiu2001} and the references therein). 
Several experiments performed by two of the authors of this paper (\cite{Falcao2014,MirandaFalcao2014b}) have shown the substantial gain in computational effort that can be achieved when using a direct quaternionic approach to this problem.
    
The Weierstrass method, also known in the literature as the Durand-Kerner method or Dochev method, is one of the most popular iterative methods for obtaining simultaneously approximations
to all the roots of a given polynomial with complex coefficients (for a survey on most of the traditional methods for root-finding we refer to \cite{McNamee2007}). The formula involved in the method was first proposed by Weierstrass \cite{Weierstrass1891}, in connection with a constructive proof of the Fundamental Theorem of Algebra, and later rediscovered and derived  in different ways by Durand \cite{Durand1960}, Dochev \cite{Dochev1962} and Kerner \cite{Kerner1966}, among others. 
	
The main purpose of this paper is to present an adaptation of the Weierstrass method to the case of quaternionic polynomials.   
By making use of the so-called Factor Theorem for quaternions we derive an iterative method which shows fast convergence and robustness with respect to the initial approximations.

The paper is organized as follows: in Section 2 we review some basic results on the algebra of real quaternions and on quaternionic polynomials;  Section~3 contains the main results of the paper; 
after revisiting the classical (complex) Weierstrass method we derive a generalization to the quaternionic case and prove, under some natural assumptions, its  quadratic order of convergence;
in Section 4 we present several numerical experiments illustrating the results obtained  in Section 3;  finally, in  Section 5 we draw some conclusions and indicate some future work.

\section{Basic results on quaternions}
In this section we present a brief summary on the main results on the algebra of real quaternions and on the ring of polynomials over the quaternions needed 
in the sequel.
\subsection{The algebra of real quaternions}

Let $\{1, \qi, \qj,\qk\}$ be an orthonormal basis of the Euclidean vector space $\R^{4}$ with a product given
according to the multiplication rules
\begin{equation*}
\qi^{2}=\qj^{2}=\qk^{2}=-1,\;\;\qi\qj=-\qj\qi=\qk.
\end{equation*}
This non-commutative product generates the well known algebra of real quaternions $\HH$. 

Given a quaternion $q= q_0+q_1\qi+q_2\qj+q_3\qk \in \HH$,
its {\textit{conjugate}} $\overline{q}$ is defined as 
$\overline{q}=q_0-q_1\qi-q_2\qj-q_3\qk$;
the number $q_0$ is called the {\textit{real}} {\textit{part}} of $q$ and denoted by $\re q$ and 
the {\textit{vector part}}   of $q$, denoted by $\V q$, is given by $\V q= q_1\qi+q_2\qj+q_3\qk$;
the {\textit{norm}} of $q$,  $|q|$, is given by
$|q|=\sqrt{q \overline{q}}=\sqrt{q_0^2+q_1^2+q_2^2+q_3^2};$
the {\textit{inverse}} of $q$ (if $q \ne 0$), denoted by $q^{-1}$ is the (unique) quaternion such that
$q q^{-1}= q^{-1} q=1$ and is given by $q^{-1} = \dfrac{\overline{q}}{|q|^2}$.

We say that a  quaternion  $q$ is {\textit{congruent}} to a quaternion $q'$, and  write  
$q\sim q'$, if there exists a non-zero quaternion $h$ such that $q'=h q h^{-1}$.
This is an equivalence relation 
in $\HH$, partitioning $\HH$ in the so-called {\textit{congruence classes}}.  We denote by $[q]$ the congruence class containing  a given quaternion $q$.
It can  be shown (see, e.g. \cite{Zhang1997}) that
\begin{equation}\label{condCongruenceClass}
 [q] =\left\{ q'\in \HH: \re q=\re q' {\text{ and }} |q|=|q'| \right\}.
\end{equation}
It follows that $[q]$ reduces to a single element  if and only if $q $  is a real number. If $q=q_0+q_1\qi+q_2\qj+q_3\qk$ is not real, its congruence class can be identified with the three-dimensional sphere in the hyperplane  $\{(x_0,x_1,x_2,x_3)\in\R^4:x_0=q_0\}$, with center $(q_0, 0, 0, 0)$ and radius $\sqrt{q_1^2+q_2^2+q_3^2}$.

\subsection{Ring of left quaternionic polynomials}

Because of the non-commutativity of quaternion multiplication, one can consider different classes of polynomials in one quaternion variable, depending on whether the variable commutes with the polynomial coefficients or not. General polynomials in the indeterminate $x$ are defined as finite sums of non-commutative monomials of the form $a_0xa_1\dots x a_j$. 
In this work we restrict our attention to polynomials whose coefficients are located only on the left-hand side of the powers of $x$, i.e. have the special form 
\begin{equation}
\label{oneSidedPols}
P(x)=a_n x^n +a_{n-1} x^{n-1} + \cdots+a_1x+a_0, \ \ a_i \in \HH; \ \  i=0,\ldots, n.
\end{equation}
These polynomials are usually called in the literature 
{\textit{one-sided}} or {\textit{unilateral}} ({\textit{left}}) polynomials. 
As usual, if $a_n \ne 0$, we will say that the {\textit{degree}} of the polynomial $P(x)$ is $n$ and  refer to $a_n$ as the leading coefficient of the polynomial. 
When $a_n=1$, we say that $P(x)$ is {\textit{monic}}.
If the coefficients $a_i$ in \eqref{oneSidedPols} are real, then we say that $P(x)$ is a {\textit{real polynomial}} and write $P(x) \in \mathbb{R}[x].$

The set of polynomials of the form \eqref{oneSidedPols} is a ring with respect to  the operations of  addition and multiplication defined as in the commutative case: for any two polynomials
$P(x)=\sum_{i=0}^n a_i x^i$ and $Q(x)=\sum_{j=0}^m  b_j x^j$,
\begin{align*}
P(x)+Q(x):=& \sum_{k=0}^{\max\{m,n\}} (a_k+b_k)x^k,\\
P(x)\ast Q(x) :=& \sum_{k=0}^{m+n} \Bigl(\sum_{j=0}^k a_j b_{k-j}\Bigr) x^k,
\end{align*}
with the implicit assumption that $a_k= 0$ for $k>n$ and $ b_k=0$ for $ k>m$.
We will denote this ring of polynomials by $\HH[x]$.
Naturally, due to the non-commutativity of the quaternionic multiplication, $\HH[x]$ is  a non-commutative ring. However, if $P(x)$ is a real polynomial, then $P(x)$ commutes with any polynomial in  $\HH[x]$.

We should also observe that the {\textit{evaluation map}} at a given quaternion $q$, defined, for the  polynomial  $P(x)$ given by \eqref{oneSidedPols}, by 
\begin{equation*}
P(q)=a_n q^n +a_{n-1} q^{n-1} + \cdots+a_1q+a_0,
\end{equation*}
is not a homomorphism from the ring $\HH[x]$ into $\HH$.
In fact, $P(x)=L(x)\ast R(x)$ does not lead, in general, to $P(q)=L(q) R(q)$.
\begin{remark}
Since all the polynomials considered will be in the indeterminate $x$, we will usually omit the reference to this variable and write simply
$P$ when referring to an element  $P(x) \in \HH[x]$, the expression  $P(q)$ being preferably reserved for the evaluation of $P$ at a specific value $q \in \HH$.
\end{remark}
We say that a quaternion $q$ is a {\textit{zero}} of a polynomial $P$, if $P(q)=0$, and   
we  use the notation ${\mathbf{Z}}_P$ to denote the {\textit{zero-set}} of $P$, i.e. the set of all the zeros of $P$.
Since this work is concerned with the computation of  zeros of polynomials, there is no loss of generality in assuming that the polynomials are monic and we will do so in what follows.

We now review some basic properties of unilateral (left) quaternion polynomials needed in the sequel. 

The next theorem shows a way of evaluating the product of two polynomials at a given quaternion, without explicitly performing their product. 
The proof of the first two results can be seen in e.g. \cite{Lam1991} and the last result is a simple consequence of the definition of the product of polynomials  and of the fact that any real number commutes with a quaternion. 

\begin{theorem}\label{teonotrootR}
Let $P= L\ast R$ with $L, R\in \HH[x]$, $q\in \HH$ and $h=R(q)$. 
\begin{enumerate}
\item[(i)] If $h=0$,  then  $P(q)=0$ (i.e. if $q$ is a zero of the right factor $R$, then $q $ is also a zero of the product  $P$).
\item[(ii)]
If $h\ne 0$, then
\begin{equation}\label{evalProduct}
P(q)=L(\tilde{q}) R(q) \quad {\text{with}} \quad \tilde{q}=h q h^{-1}.
\end{equation}
In particular, if $q$  is a zero of $P$ which is not a zero of $R$, then 
$\tilde{q}$ is a zero of $L$. 
\item[(iii)]
If $L\in \mathbb{R}[x]$, then 
\begin{equation}\label{evalProductReal}
P(q)= R(q)L(q).
\end{equation}
\end{enumerate}
\end{theorem}
The following result, first proved  by Gordon and Motzkin  \cite{GordonMotzkin1965}, can also be seen in  \cite{Lam1991}.
\begin{theorem}[Factor Theorem]\label{FactorTheorem}
Let $P \in \HH[x]$  and $q \in \HH$. Then, $q$ is a zero of $P$ if and only if there exists $Q \in \HH[x]$ such that
\begin{equation*}
P(x)=Q(x)\ast (x-q).
\end{equation*}
\end{theorem}

In 1941, Niven \cite{Niven1941} proved the Fundamental Theorem of Algebra for unilateral quaternionic polynomials, establishing that any non-constant polynomial in $\HH[x]$ always has a zero in $\HH$. 
More general results are contained in the following theorem.
\begin{theorem}\label{theoremResultsZeros}
Let $P$ be a  monic polynomial of degree $n\, (n\ge 1)$ in $\HH[x]$. 
Then:
\begin{enumerate} 
\item[(i)]
$P$ admits a factorization into linear factors, i.e. there exist $x_1,\ldots, x_n\in \HH$, such that
\begin{equation*} 
P(x)=(x-x_n)\ast(x-x_{n-1})\ast \cdots\ast  (x-x_1).
\end{equation*}
\item[(ii)]
For the factor terms $x_i$ referred in (i), we have:
\begin{enumerate}
\item[(a)]
${\mathbf{Z}}_P\subseteq \displaystyle{\bigcup_{i=1}^n [x_i]}$.
\item[(b)] Each of the congruence classes $[x_i]; i=1,\ldots,n,$ contains (at least) a zero of $P$.
\end{enumerate}
\item[(iii)]
If 
\begin{equation*}
P(x)= (x-y_n)  \ast(x-y_{n-1})\ast\cdots  \ast (x-y_1)
\end{equation*}
is another factorization of $P$ into linear factors,  then there exists a permutation $\pi$ of $(1,2,\ldots,n)$ and  
$h_i\in \HH; i=1,\dots,n$, such that
\begin{equation*}
y_{\pi(i)}=h_i x_i h_i^{-1}.
\end{equation*}
\end{enumerate}
\end{theorem}
The first result in the above theorem is an immediate consequence of the Fundamental Theorem of Algebra for quaternion polynomials and of the Factor Theorem; the proof of the other results can be found in \cite{Lam1991} and \cite{SerodioSiu2001}.

Given a polynomial $P(x)=\sum_{k=0}^n a_k x^k$,  its {\textit{conjugate polynomial}}, denoted by $
\overline{P}(x)$, is given by
\begin{equation*}
\overline{P}(x)=\sum_{k=0}^n {\overline{a_k}} x^k.
\end{equation*}

It is very simple to verify that, for all $ P,Q\in \HH[x]$:
\begin{align*}
\overline{P\ast Q}&= \overline{Q}\ast \overline{P}, \\
\label{PPbar}
P\ast \overline{P} &\in \mathbb{R}[x]  {\text{\ \  and\ \  }}  P \ast \overline{P}=\overline{P} \ast P.
\end{align*}

To each quaternion $q$, we will 
associate the following polynomial
\begin{equation*}
{\cal Q}_q(x) :=
(x-q)\ast(x-\overline{q}) =x^2 - 2 \re q \, x + |q|^2,
\end{equation*}
called the {\textit{characteristic polynomial}} of $q$.
Since the characteristic polynomial of $q$ only depends on the real part and norm of $q$ and
recalling  \eqref{condCongruenceClass}, we immediately conclude that
$
{\cal Q}_q={\cal Q}_{q'}$ if and only if $ [q]=[q']$.
Note that ${\cal Q}_q$ is a quadratic polynomial with real coefficients.
It can also be shown that the zero-set of ${\cal Q}_q$ is the congruence class of $q$, i.e. 
$\mathbf{Z}_{{\cal Q}_q} = [q]$; see, e.g. \cite{Zhang1997}.
This result already shows that,  in what concerns the number of zeros,  polynomials in $\HH[x]$ can behave very differently from complex polynomials: a polynomial in $\HH[x]$ can have an infinite number of zeros. However, as Theorem~\ref{theoremResultsZeros} shows, the zeros of a polynomial of degree $n$  belong to, at most,  $n$ congruence classes in $\HH$.
 
The zeros of an unilateral  quaternionic polynomial can be of two distinct types, the so-called {\textit{isolated zeros}} and {\textit{spherical zeros}}, whose definitions  we now recall.
Let $q$ be a zero of a given polynomial $P$. We say that $q$ is an {\textit{isolated zero}} of $P$ if the congruence class of $q$ contains no other  zero of $P$.
If $q$ is not an isolated zero of $P$, we call it a {\textit{spherical zero}} of $P$.
Note that, according to the definition, real zeros are always isolated zeros. 
The next theorem gives conditions under which a non-real zero is a spherical zero (see e.g.~\cite{PogoruiShapiro2004}).
\begin{theorem}\label{teorSphericalRoots}
Let $q$ be a non-real zero of  a given polynomial $P\in \HH[x]$. Then, $q$ is a spherical zero of $P$ if and only if any of the following 
equivalent conditions hold:
\begin{enumerate}
\item[(i)]
$q$ and $\overline{q}$ are both zeros of $P$.
\item[(ii)]
$[q] \subseteq {\mathbf{Z}}_P$.
\item[(iii)]
The characteristic polynomial of $q$, ${\cal  Q}_q$, is a divisor of $P$, i.e. there exists a polynomial $Q \in \HH[x]$ such that
$
P=Q \ast {\cal  Q}_q.
$ 
\end{enumerate}
\end{theorem}
Recalling that the congruence classes of non-real quaternions can be identified with spheres, condition (ii) justifies the choice of the term {\textit{spherical}} to designate this type of zeros. When $q$ is a spherical zero, we also say that $q$  {\textit{generates  the sphere of zeros}} $[q]$.
 
\section{The Weierstrass method in $\mathbb{H}[x]$}

Let $P$ be a complex monic polynomial of degree $n$ with roots $\zeta_1, \ldots,\zeta_n$  and let $z_1^{(0)}, \ldots, z_{n}^{(0)} $ 
 be $n$ given distinct numbers.
The  (classical) {\textit{Weierstrass method}} for approximating the roots $\zeta_i$ is defined by the iterative scheme:
\begin{equation}\label{classicalWParallel}
z_i^{(k+1)}=z_i^{(k)}-\frac{P(z_i^{(k)})}{\displaystyle{\prod_{\genfrac{}{}{0pt}{}{j=1}{j\ne i}}^n} (z_i^{(k)}-z_j^{(k)})}; i=1, \ldots, n; \ k=0,1,2,\ldots
\end{equation}
If the roots $\zeta_1, \ldots,\zeta_n$  are distinct and $z_1^{(0)}, \ldots, z_{n}^{(0)} $  are sufficiently good initial approximations
to these roots, then the method converges at a quadratic rate, as was firstly proven by Dochev \cite{Dochev1962} (see also \cite{Aberth1973,Werner1982}). For multiple roots, the method still converges (locally) but the quadratic convergence is lost; see e.g. \cite{Fraigniaud1991}.

Formula \eqref{classicalWParallel} is realized in {\textit{parallel mode}} and is often called the {\textit{total-step} mode}.
The convergence of the method can be accelerated by using a different variant that makes use of  the  most recent updated approximations to the roots as soon as they are available, as follows:
\begin{equation*}
z_i^{(k+1)}=z_i^{(k)}-
\frac{P(z_i^{(k)})}
{\displaystyle{\prod_{j=1}^{i-1} (z_i^{(k)}-z_j^{(k+1)})}
\displaystyle{\prod_{j=i+1}^n (z_i^{(k)}-z_j^{(k)})
}}; i=1, \ldots, n; \ k=0,1,2,\ldots
\end{equation*}
The above variant of the Weierstrass method is usually referred to as the {\textit{serial}}, {\textit{sequential}} or {\textit{single-step}} mode (see \cite{PetkovicTrajkovic1994} and references therein).
  
\subsection{A quaternionic Weierstrass-like scheme}
Our purpose is to adapt the idea of the  Weierstrass method to the computation of the zeros of quaternionic  polynomials.
So, let $P$ be a given monic polynomial of degree $n$ in $\HH[x].$
Corresponding to the assumption imposed in the complex case to guarantee the quadratic convergence of the  method -- i.e. that the zeros of the polynomial are simple --  we  will now assume that the polynomial $P$ has $n$ distinct isolated roots.
By analogy with the complex case, in this situation, we will still  say that $P$ has only {\textit{simple roots.}}
As stated in the previous section, $P$ can be factorized in the form
\begin{equation}\label{fatPolP}
P(x)=(x-x_n) \ast (x-x_{n-1}) \ast \dots \ast (x-x_1),
\end{equation}
with the factor terms $x_i \in \HH.$ 
For simplicity, we introduce the following convenient notation, which we borrow and adapt from \cite{GentiliStoppatoStruppa2013},
\begin{equation*} 
\prodstar_{i=k}^m (x-\alpha_i):=(x-\alpha_m)\ast (x-\alpha_{m-1})\ast \cdots  \ast(x-\alpha_k).
\end{equation*}
\begin{remark}
Note that the order of the factors, due to the non-commutativity of the product in $\HH[x]$,  is important. We also adopt the convention that
\begin{equation*}
\prodstar_{i=k}^m (x-\alpha_i):=1,  \quad{\text{whenever\ }}\ k>m.
\end{equation*}
\end{remark}

We first present a simple lemma, relating the roots of $P$ with the quaternions involved in any of its factorizations.
\begin{lemma}
Let $P $ be a (monic) polynomial of degree $n$ in  $ \HH[x]$ with simple roots and let \eqref{fatPolP} be one of its factorizations. Then:

\begin{enumerate}
\item[(i)] The congruence classes of the elements $x_j;j=1,\dots,n,$ in \eqref{fatPolP} are distinct.

\item[(ii)]
 The roots $\zeta_1, \ldots, \zeta_n$ of $P$ can be obtained from the quaternions $x_1, \ldots, x_n$ as follows:
\begin{equation}\label{formulaRoots}
\zeta_i=\overline{R_i}(x_i)\, x_i\, \bigl(\overline{R_i}(x_i)\bigr)^{-1}; \ i=1,2, \ldots, n,
\end{equation}
where $R_i$ are the polynomials given by
\begin{equation}\label{polRi}
 R_i:=\prodstar_{j=1}^{i-1} (x-x_j).
\end{equation}
\end{enumerate}
\end{lemma}
\begin{proof}
The fact that the congruence classes $[x_j]; j=1, \ldots, n$, are distinct  is an immediate consequence of the results in  Theorem~\ref{theoremResultsZeros} and of the assumption that $P$ has only simple roots, i.e.~it has exactly $n$ isolated roots.
The proof that the roots of $P$ are given by  \eqref{formulaRoots} is a simple adaptation of the proof of \cite[Proposition~16.3]{Lam1991}.
\end{proof}

Following the idea of the  Weierstrass method in its sequential version, we will now show how to obtain sequences converging, at a quadratic rate, to the factor terms in~\eqref{fatPolP} of a given polynomial $P$.
Then, we will show  how these sequences can be used to estimate the zeros of $P$. 
\begin{theorem}\label{TeorEsquema}
Let $P$ be a  polynomial of degree $n$ in $\HH[x]$ with simple roots and, for $ i=1,\ldots, n;$ $ k=0,1,2,\ldots$, let 
\begin{equation} \label{scheme}
z_i^{(k+1)} =  z_i^{(k)} - \left( \overline{ {\cal L}_i^{(k)} } \ast P \ast \overline{ {\cal R}_i^{(k)}} \right)( z_i^{(k)} ) \,\left({\boldsymbol{\cal Q}}^{(k)}_i(z_i^{(k)}) \right)^{-1},
\end{equation}
where
\begin{align}
{\cal L}_i^{(k)}(x):=& \prodstar_{j=i+1}^{n}  \big(x-z_j^{(k)}\big), \label{polCalL} \\
{\cal R}_i^{(k)}(x):=& \ \prodstar_{j=1}^{i-1} \big(x-{z}_j^{(k+1)}\big)\label{polCalR}\\
\intertext{and}
\label{calQ}
{\boldsymbol{\cal Q}}^{(k)}_i(x):=&
\prod_{j=1}^{i-1} {\cal Q}_{z_j^{(k+1)}}(x)\prodstar  \prod_{j=i+1}^{n}  {\cal Q}_{z_j^{(k)}}(x),
\end{align}
with ${\cal Q}_q$ denoting the characteristic polynomial of $q$.
If the initial approximations $z_i^{(0)}$ are sufficiently close to the factor terms $x_i$ in a factorization of $P$ in the form \eqref{fatPolP}, then the sequences $\{z_i^{(k)}\}$ 
converge quadratically to $x_i$.
\end{theorem}
\begin{proof}
Let $z_i^{(k)}$ be approximations to $x_i$
 with errors $\varepsilon_i^{(k)}$, i.e. 
\begin{equation}\label{errors}
\varepsilon_i^{(k)}:=x_i-z_i^{(k)}, \ i=1,\dots,n,
\end{equation}  
and let 
\begin{equation*}
\varepsilon^{(k)} := \max_i |\varepsilon_i^{(k)}|.
\end{equation*}
We assume that $\varepsilon^{(k)}$ is {\textit{small enough}}, i.e. that $z_i^{(k)}$ are  {\textit{sufficiently good}} approximations  to $x_i$. We want to show that the next iterates $z_i^{(k+1)}$ are approximations to $x_i$ with errors $\varepsilon_i^{(k+1)}$ such that
\begin{equation*}
\varepsilon_i^{(k+1)}={\cal O}\bigl((\varepsilon^{(k)})^2\bigr).
\end{equation*}
We will do this by induction on $i$.
 For simplicity, we will  omit the iteration superscript $(k)$, writing simply $z_i$ for $z_i^{(k)}$, $\varepsilon_i$ for $\varepsilon_i^{(k)}, {\cal L}_i$ for ${\cal L}_i^{(k)}$ etc. and will replace the superscript $(k+1)$ by  a tilde symbol, using  $\tilde{z_i}$ for $z_i^{(k+1)}$, $\tilde{\varepsilon_i}$ for $\varepsilon_i^{(k+1)}$, etc. 

\noindent 
{\bf Step 1:} We first prove that the result is true for $i=1$, i.e. that we have $\tilde{\varepsilon}_1={\cal O}(\varepsilon^2).$

By making use  of \eqref{errors}, we can rewrite the polynomial
$P(x)$ as
\begin{align*}
P(x)& =\prodstar_{j=1}^n (x-x_j) =
\prodstar_{j=2}^{n} (x-z_j-\varepsilon_j)\ast (x-z_1-\varepsilon_1) \\ 
&=\Bigl(\prodstar_{j=2}^{n} (x-z_j)+\mathscr{E}_1(x)\Big)\ast (x-z_1-\varepsilon_1), 
\end{align*}
where  $\mathscr{E}_1(x)$ 
designates  a  remainder polynomial consisting  of a sum of  $n-1$ terms of the form
\begin{equation*}
-(x-z_n)\ast\dots\ast(x-z_{j-1})\ast \varepsilon_j \ast (x-z_{j+1}) \ast \dots \ast(x-z_2),
\end{equation*}
($j= 2,\ldots,n $)  with  terms with  $\ast$-products involving at least two $\varepsilon_j$'s. 
By using the definition \eqref{polCalL} of the polynomial ${\cal L}_1$,
we can write  $P(x)$ in the following form
\begin{align*}
P(x)&=\Bigl({\cal L}_1(x)+ \mathscr{E}_1(x)\Bigr) \ast (x-z_1-\varepsilon_1) \nonumber\\
&= 
{\cal L}_1(x) \ast (x-z_1-\varepsilon_1) + \mathscr{E}_1(x) \ast (x-z_1-\varepsilon_1).
\end{align*}
Let $\overline{{\mathcal L}}_1$ be the conjugate of $ {\cal L}_1$ and
note that 
$\overline{{\cal L}}_1 \ast{\cal L}_1$ is precisely the real polynomial ${\boldsymbol{\cal{Q}}}_1$ defined by \eqref{calQ}.
Hence, if  we multiply $P(x)$ on the left by $\overline{{\cal L}}_1$ and evaluate the resulting polynomial at the point $x=z_1$, we obtain, 
 recalling the results  \eqref{evalProduct} and \eqref{evalProductReal} in 
 Theorem~\ref{teonotrootR}, 
\begin{align*}
\bigl(\overline{{\cal L}}_1\ast P\bigr)(z_1) 
 = -\varepsilon_1 {\boldsymbol{\cal{Q}}}_1(z_1)
  - \Bigl(\overline{{\cal L}}_1\ast  \mathscr{E}_1\Bigr)(\hat{z}_1)\, \varepsilon_1, 
\end{align*}
where $\hat{z}_1=\varepsilon_1 z_1 \varepsilon_1^{-1}$.
Observing that we may assume that we are working in a bounded domain ${\cal D}$ of $\HH$ (a sufficiently large disk containing all $z_i$) and recalling the definition of 
$\mathscr{E}_1$,
it is easily seen that we have
\begin{equation*}
\mathscr{E}_1 (\alpha)={\cal O}(\varepsilon),\ \forall \alpha \in {\cal D} 
\end{equation*}
and therefore 
\begin{equation*}
\bigl(\overline{{\cal L}}_1\ast P\bigr)(z_1) = -\varepsilon_1 {\boldsymbol{\cal{Q}}}_1(z_1)+ {\cal O}(\varepsilon^2).
\end{equation*}
Since we are assuming that the congruence classes $[x_j]$ are distinct, then, for sufficiently small $\varepsilon$, $|{\boldsymbol{\cal{Q}}}_1(z_1)|$ is bounded away from zero and so, by multiplying both sides of the above equality on the right by $\bigl({\boldsymbol{\cal{Q}}}_1(z_1)\bigr)^{-1}$, we obtain
\begin{equation*}
\left(\overline{{\cal L}}_1\ast P\right)(z_1)\, \left({\boldsymbol{\cal{Q}}}_1(z_1)\right)^{-1}=-\varepsilon_1+{\cal O}(\varepsilon^2),
\end{equation*}
or, in other words (cf. \eqref{errors}),
\begin{equation*}
x_1=z_1- \left(\overline{{\cal L}}_1\ast P\right)(z_1)\left({\boldsymbol{\cal{Q}}}_1(z_1)\right)^{-1}+{\cal O}(\varepsilon^2),
\end{equation*}
which means that the next approximation to $x_1$ 
\begin{equation*}
\tilde{z}_1=z_1-(\overline{{\cal L}}_1\ast P)(z_1)\, \bigl({\boldsymbol{\cal{Q}}}_1(z_1)\bigr)^{-1}
\end{equation*}
is such that 
\begin{equation*}
\tilde{\varepsilon}_1=x_1-\tilde{z}_1={\cal O}(\varepsilon^2).
\end{equation*}
{\bf Step  $\boldsymbol{i}$:} 
We now assume that, for $j=1, \ldots, i-1$, ${\tilde{z}_j}$ approximates  $x_j$ with an error $\tilde{\varepsilon}_j$  such that 
$\tilde{\varepsilon}_j={\cal O}(\varepsilon^2)$ and 
prove that ${\tilde{z}_i}$ is also an ${\cal O}(\varepsilon^2)$ approximation  to $x_i$.

Using the polynomials 
\begin{equation*}
L_i(x)=\prodstar_{j=i+1}^n (x-x_j)
\quad\text{and}\quad
R_i(x)= \prodstar_{j=1}^{i-1}(x-x_j)
\end{equation*}
we can write 
\begin{align}
\label{Li}
L_i(x) &= \prodstar_{j=i+1}^n (x-z_j-\varepsilon_j)
       = \prodstar_{j=i+1}^n (x-z_j)+ \mathscr{E}_i(x)
       = {\cal L}_i(x)+\mathscr{E}_i(x)\\
\intertext{and}
\label{Ri}
R_i(x) &= \prodstar_{j=1}^{i-1} (x-\tilde z_j-\tilde \varepsilon_j)
       = \prodstar_{j=1}^{i-1} (x-\tilde z_j)+ \tilde{\mathscr{E}}_i(x)
       = {\cal R}_i(x)+\tilde{\mathscr{E}}_i(x),
\end{align}
where  $\mathscr{E}_i$ and  $\tilde{ \mathscr{E}}_i$ are remainder polynomials defined in an analogous manner to  $\mathscr{E}_1$, with the obvious modifications. Note that $\mathscr{E}_i$ is a sum of terms, all of which involve at least the product by a $\varepsilon_{j}$ ($j \in \{i+1, \ldots, n\}$) and $\tilde{\mathscr{E}}_i $ a sum of terms, all of which involve at least the product by an $\tilde{\varepsilon}_{j}$ ($j \in \{1, \ldots, i-1\}$).
Therefore 
\begin{equation*}
\mathscr{E}_i (\alpha)={\cal O}(\varepsilon)
\quad \text{and}\quad 
\tilde{\mathscr{E}}_i (\alpha)={\cal O}(\varepsilon^2),\ 
 \forall \alpha \in {\cal D}.
\end{equation*}
Hence the polynomial $P$ can be written as 
\begin{align*}
P(x) & =  L_i(x)\ast (x-x_i)\ast R_i(x)\\
     & = \big({\cal L}_i(x) +  \mathscr{E}_i(x)  \big)   \ast (x-z_i-\varepsilon_i)  \ast  \big({\cal R}_i(x)  + \tilde{ \mathscr{E}}_i(x)   \big).
\end{align*}

Multiplying both sides of the last equality on the left by $\overline{{\cal L}}_i$ and on the right by  $\overline{{\cal R}}_i$ and evaluating at $x=z_i$, 
 we obtain
\begin{align*}
\bigl(\overline{{\cal L}}_i\ast P \ast \overline{{\cal R}}_i\bigr)(z_i)
&=
\Bigl( 
\overline{{\cal L}}_i 
\ast {\cal L}_i  \ast {\cal R}_i 
\ast 
\overline{{\cal R}}_i \ast (x-z_i -\varepsilon_i)
\Bigr)(z_i)\\
& \qquad +
\Bigl(\overline{{\cal L}}_i  \ast  {\cal R}_i
\ast 
 \overline{{\cal R}}_i
\ast{\mathscr{E}}_i \ast (x-z_i -\varepsilon_i) 
\Bigr)(z_i)\\
& \qquad +
\Bigl(
\overline{{\cal L}}_i  
\ast {\cal L}_i  \ast (x-z_i -\varepsilon_i)  \ast \tilde{\mathscr{E}}_i 
\ast  
\overline{{\cal R}}_i
\Bigr)(z_i)\\
& \qquad +
\Bigl(
\overline{{\cal L}}_i
 \ast {\mathscr{E}}_i \ast (x-z_i -\varepsilon_i) \ast \tilde{{\mathscr{E}}}_i
 \ast 
  \overline{{\cal R}}_i
\Bigr)(z_i),
\end{align*}
where we made use of the fact that ${\cal R}_{i}\ast \overline{{\cal R}}_{i}$ is a real polynomial and hence commutes with any other polynomial. 
Observing that $ \overline{{\cal L}}_i 
\ast {\cal L}_i  \ast {\cal R}_i 
\ast 
\overline{{\cal R}}_i $ is the real polynomial ${\boldsymbol{\cal Q}}_i$,  using again the results  \eqref{evalProduct} and \eqref{evalProductReal} in 
 Theorem~\ref{teonotrootR} and having in mind the form of the remainder polynomials ${\mathscr{E}}_i$ and 
 $\tilde{\mathscr{E}}_i$, we can write
 \begin{align} \label{LPR}
 \bigl(\overline{{\cal L}}_i\ast P \ast \overline{{\cal R}}_i\bigr)(z_i)
& =
 -\varepsilon_i {\boldsymbol{\cal{Q}}}_i(z_i)-
 \bigl( \overline{{\cal L}}_i  \ast  {\cal R}_i
 \ast 
  \overline{{\cal R}}_i
 \ast{\mathscr{E}}_i \bigr)(\hat{z}_i) \varepsilon_i+
 {\cal O}(\varepsilon^2)
\nonumber \\
 &=
 -\varepsilon_i {\boldsymbol{\cal{Q}}}_i(z_i)+ {\cal O}(\varepsilon^2),
 \end{align}
 where $\hat{z}_i=\varepsilon_i z_i \varepsilon_i^{-1}$.
 Multiplying by $\left({\boldsymbol{\cal Q}}_i(z_i)\right)^{-1}$ on the right and observing, once more, that $|\bs{{\cal Q}}_i(z_i)|$ is bounded away from zero, we obtain 
\begin{align*}
\left(\overline{{\cal L}}_i \ast P \ast \overline{{\cal R}}_i \right)(z_i)\,\left({\boldsymbol{\cal Q}}_i(z_i)\right)^{-1} =-\varepsilon_i+{\cal O}(\varepsilon^2)
\end{align*}
or, equivalently, recalling the definition of the errors $\varepsilon_i$,
\begin{align*}
\left(\overline{{\cal L}}_i \ast P \ast \overline{{\cal R}}_i \right)(z_i)\,\left({\boldsymbol{\cal Q}}_i(z_i)\right)^{-1} =z_i-x_i+{\cal O}(\varepsilon^2)
\end{align*}
showing that 
\begin{equation*}
\tilde{z}_i=z_i-\left(\overline{{\cal L}}_i \ast P \ast \overline{{\cal R}}_i \right)(z_i) \, \left({\boldsymbol{\cal Q}}_i(z_i)\right)^{-1} 
\end{equation*}
is an ${\cal O}(\varepsilon^2)$ approximation to $x_i$, 
which is precisely the result that we wanted to establish.
\end{proof}

\begin{remark}\label{serial}
We should observe that, for each $i=1, \ldots, n$,  formula \eqref{scheme} for the computation of the 
approximation $z_i^{(k+1)}$  to $x_i$ involves the polynomials ${\cal R}_i^{(k)}$ and  
${\boldsymbol{\cal Q}}^{(k)}_i$ which make use of the already computed $z_1^{(k+1)}, \ldots, z_{i-1}^{(k+1)}$, i.e. the method here described can be seen as a generalization of the sequential version of the Weierstrass method. 
A careful analysis of the proof, namely the deduction of formula \eqref{LPR}, shows that the use of the updated $z_j^{(k+1)}; j=1, \ldots, i-1,$ when computing $z_i^{(k+1)}$, is essential for establishing the quadratic order of convergence of the method.
\end{remark}

We now show how, with some additional little effort, one can use the iterative scheme \eqref{scheme}--\eqref{calQ} to produce, not only the factor terms, but also the roots of the polynomial.
\begin{theorem}\label{main}
Let $P$ be a monic polynomial  of degree $n$ in $ \HH[x]$ with simple roots and let  $\{z_i^{(k)}\}$ be the  sequences  defined  by  the Weierstrass  iterative scheme \eqref{scheme}--\eqref{calQ} under the assumptions of Theorem~\ref{TeorEsquema}. Finally, let  $\{\zeta_i^{(k)}\}$ be the sequences defined by 
\begin{equation}
\label{zetak}
\zeta_i^{(k+1)}:=\overline{ {\cal R}_i^{(k)}}( z_i^{(k+1)} )\, z_i^{(k+1)}\, \Bigl( \overline{ {\cal R}_i^{(k)}} ( z_i^{(k+1)} )\Bigr)^{-1};\, k=0,1,2, \ldots,
\end{equation}
where ${\cal R}_i^{(k)}$ are the polynomials given by \eqref{polCalR}.
Then,
$\{\zeta_1^{(k)}\}, \ldots, \{\zeta_n^{(k)}  \}$ converge quadratically to the roots  of $P$.
\end{theorem}
\begin{proof}
We start by first recalling that the roots $\zeta_i$; $i=1,\ldots,n,$ of $P$ are related to $x_i$ in \eqref{fatPolP} through \eqref{formulaRoots}, i.e. 
\begin{equation*}
\zeta_i=\overline{R_i}(x_i)\, x_i\, \bigl(\overline{R_i}(x_i)\bigr)^{-1}; \ i=1, \ldots, n,
\end{equation*}
with $R_i$ defined by \eqref{polRi}.

Next, denote by   
$\varepsilon_i^{(k+1)}$ the errors in the approximations $z_i^{(k+1)}$ to $x_i$
and let
$\varepsilon^{(k+1)}:=\max_i |\varepsilon_i^{(k+1)}|$.
We will show that 
\begin{equation*}
\zeta_i^{(k+1)}=\zeta_i+{\cal O}(\varepsilon^{(k+1)}).
\end{equation*}
This, conjugated with the results of Theorem~\ref{TeorEsquema}, will prove the assertion of the theorem.

Similarly to what we did in the proof of Theorem~\ref{TeorEsquema}, we will simply write $\tilde{z}_i$ for $z_i^{(k+1)}$, $\tilde{\varepsilon}_i$ for $\varepsilon_i^{(k+1)}$, $\tilde{\varepsilon}$ for $\varepsilon^{(k+1)}$, $\tilde{\zeta}_i$ for $\zeta_i^{(k+1)}$ and ${\cal R}_i$ for ${\cal R}_i^{(k)}$.
Taking into account that the polynomials $R_i$ in \eqref{polRi} are exactly the same polynomials presented in \eqref{Ri}, we can write
\begin{equation*}
\overline{R_i}(x) = \overline{{\cal R}_i}(x)+\overline{\tilde{\mathscr{E}}}_i(x)
\end{equation*}
and therefore
\begin{equation*}
\overline{R_i}(x_i)=\overline{\cal{R}}_i(x_i)+{\cal{O}}(\tilde{\varepsilon}). 
\end{equation*}

Expressing $\overline{{\cal R}_i}(x)$ in the expanded form $\sum_{j=1}^{i-1}\bar r_j x^j$, it follows at once that 
\begin{equation}\label{Rixi}
\overline{R_i}(x_i)
= \overline{\cal{R}}_i(\tilde{z}_i+\tilde{\varepsilon}_i)+{\cal O}(\tilde{\varepsilon})  
 = \overline{\cal{R}}_i(\tilde{z}_i)+{\cal O}(\tilde{\varepsilon}).
\end{equation}
Combining the fact that both $|\overline{R_i}(x_i)|$ and $|\overline{\cal{R}}_i(\tilde{z}_i|$ are bounded away from zero with the result \eqref{Rixi}, we can conclude that 
\begin{equation}\label{RixiInv}
\bigl(\overline{R_i}(x_i)\bigr)^{-1}
= \left( \overline{\cal{R}}_i(\tilde{z}_i) \right)^{-1}+{\cal O}(\tilde{\varepsilon}). 
\end{equation}
Finally, result \eqref{formulaRoots} together with \eqref{Rixi}, \eqref{RixiInv} and the assumption \eqref{zetak} gives
\begin{align*}
\zeta_i  & = \overline{R_i}(x_i)\, x_i\,  \bigl(\overline{R_i}(x_i)\bigr)^{-1} \nonumber \\
&  =\Bigl(\overline{{\cal R}_i} (\tilde{z}_i)+{\cal O}(\tilde{\varepsilon})\Bigr) \bigl(\tilde{z}_i+ \tilde{\varepsilon}_i \bigr)
 \Bigl( \bigl(\overline{ {\cal R}_i}(\tilde{z}_i)\bigr)^{-1}+{\cal O}(\tilde{\varepsilon})\Bigr)\nonumber \\
 & = \overline{{\cal R}_i}(\tilde{z}_i)\, \tilde{z}_i\, \bigl(\overline{{\cal R}_i}(\tilde{z}_i)\bigr)^{-1} +{\cal O}(\tilde{\varepsilon})\\ \nonumber
 & =\tilde{\zeta}_i+{\cal O}(\tilde{\varepsilon}),
\end{align*}
which is precisely the result we want to prove.
\end{proof}

\subsection{Computational details}
We now summarize the proposed algorithm for computing the roots of a given  quaternionic unilateral polynomial $P$ of degree $n$ and make some  practical comments regarding its implementation.

\subsubsection*{Quaternionic-Weierstrass algorithm}

{\sc Input:}
\begin{itemize}
\item[-] \texttt{polynomial coefficients}
\item[-] \texttt{initial values $z_i^{(0)}$}
\item[-] \texttt{error tolerances $\varepsilon_1,\varepsilon_2$}
\item[-] \texttt{maximum number of iterations $kmax$}
\end{itemize}
\begin{enumerate}
\item \texttt{Set $\zeta_i^{(0)}=z_i^{(0)}$}\\
\item  \texttt{For $k=1,2,\ldots$ until Stopping Criterion is true}\\
\begin{enumerate}
\item \texttt{Compute $z_i^{(k)}$, by means of \eqref{scheme}-\eqref{calQ}}.\\
\item \texttt{Compute $\zeta_i^{(k)}$,  by means of \eqref{zetak} and  \eqref{polCalR}}.\\
\end{enumerate}
\texttt{Stopping Criterion:\\
\Big($\displaystyle{\max_i} \big| \zeta_i^{(k)}-\zeta_i^{(k-1)}\big| < \varepsilon_1$ and
 $\displaystyle{\max_i} \big| P( \zeta_i^{(k)}) \big| < \varepsilon_2\Big)$
  or $k=kmax$}.
\end{enumerate}
{\sc Ouput:} \texttt{Factors} $\tilde x_i=z_i^{(k)}$ \texttt{and roots} $\tilde\zeta_i=\zeta_i^{(k)}$.

\subsubsection*{Choice of initial approximations}

In the classical case,  Weierstrass method seems in practice to converge from nearly all starting points (see \cite{McNamee2007} and the references therein for details). The numerical experiments that we have conducted also show the robustness of the quaternionic version of the method in what concerns the choice of initial approximations.  In any case, there are some aspects that should be taken into account.

First, for  formula \eqref{scheme} to be meaningful,  a first requirement one has to have in mind when choosing the initial approximations $z_1^{(0)}, \ldots,z_{n}^{(0)}$ is that all of them  belong to distinct congruence classes.
This does not necessarily guarantee that, in the course of the computations, two approximations do not fall into the same congruence class, although this is very unlikely to happen. In such a case, a small perturbation of the initial guesses should be sufficient to regain convergence.

Second, it is, naturally, convenient to select the initial approximations from a region where the $x_i$ in any factorization of the polynomial $P$ are known to lie.  Since the $x_i$ and the roots $\zeta_i$ of $P$ have the same norm, bounds on $|\zeta_i|$ are also valid for $|x_i|$. 
Moreover, since $P(x)\ast\overline{P}(x)$ is a real polynomial\footnote{The use of the polynomial $P(x)\ast\overline{P}(x)$ goes back to the work of Niven \cite{Niven1941}.} 
, whose roots $r_i$ also have the same norm as the roots $\zeta_i$ of $P$ (this is an immediate consequence of Theorem~\ref{FactorTheorem} in \cite{SerodioSiu2001} and the characterization of the congruence classes given by \eqref{condCongruenceClass}), one can use any known result on bounds on (complex) polynomial roots to obtain a region from where the initial approximations should be selected.
 
\subsubsection*{Non simple zeros}

The proof of Theorem~\ref{TeorEsquema} was done under the assumption that the roots $\zeta_1,\dots,\zeta_n$ of the polynomial 
\eqref{fatPolP} are simple, i.e. that $[\zeta_i]\ne [\zeta_j]$ for all $i \ne j$, or, equivalently, $[x_i]\ne[x_j]$
(cf. Lemma~\ref{fatPolP}). When $[x_i]=[x_j]$ for some $i \ne j$, the characterization of the zero-set of the polynomial
can be done taking into account the following two results.

\begin{lemma}\label{lemmaPermute2Factors}
If $x_1,x_2 \in \HH$ and $h=\overline{x}_2-x_1$, then
\begin{equation*}
(x-x_2)\ast(x-x_1)=
\begin{cases}
(x-h^{-1}x_1 h) \ast (x-h^{-1} x_2 h), & \text{if }h\neq0\\
(x-x_1)\ast(x-x_2), & \text{if }h=0.
\end{cases}
\end{equation*}
\end{lemma}
\begin{proof}
The result follows by simple manipulation; see also \cite{SerodioSiu2001} for a different, but equivalent result.
\end{proof}

\begin{lemma}\label{lemmaFactQuadratic}
Consider a quadratic polynomial factorized in the form
\begin{equation*}
P(x)=(x-x_2)\ast(x-x_1),
\end{equation*}
where $x_1,x_2 \in \HH \setminus \R$ and  $[x_1]=[x_2]$.

\begin{enumerate}
\item[(i)]  If $x_1 \ne \overline{x}_2$, then the only zero of $P$ is $x_1$.

\item[(ii)]  If $x_1 = \overline{x}_2$, then $x_1$ generates the sphere of zeros $[x_1] $, i.e. $x_1$ is a spherical zero.
\end{enumerate}
\end{lemma}
\begin{proof}
See e.g. \cite{GentiliStoppatoStruppa2013}.
\end{proof}

The problem of finding a natural definition of multiplicity for zeros of quaternionic polynomials is a rather complicated task and as a consequence one can find in the literature different (not always equivalent, see \cite{FalcaoMirandaSeverinoSoares2017}) concepts of multiplicity \cite{Beck1979,Bolotnikov2015,GentiliStoppato2008b,GentiliStruppa2008,PereiraVettori2006,Topuridze2009}. In the case (i) above, we will say that $x_1$ is a root with (\textit{isolated}) \textit{multiplicity} equal to two.
For example, the polynomials $(x+1-\qi)\ast (x+1+\qk)$ and $(x+1+\qk)\ast (x+1+\qk)$ both have $\zeta=-1-\qk$ as a root with multiplicity two.

Returning to the case of a general polynomial of degree $n$ of the form \eqref{fatPolP}, we consider, for simplicity, that 
$[x_i]$ and $[x_j]$ are the only non-distinct congruent classes. Using Lemma~\ref{lemmaPermute2Factors}, we can freely move the factors $(x-x_i)$ and $(x-x_j)$ to the right of the factorization without changing the set of congruence classes so that in the new factorization 
\begin{align*}
P(x)&=(x-y_n) \ast (x-y_{n-1}) \ast \dots \ast (x-y_2) \ast (x-y_1)\\
&=Q(x) \ast (x-y_2)\ast(x-y_1) 
\end{align*}
we have $[y_1]=[y_2]$. Observe that all  the roots of the $n-2$ degree polynomial $Q$ are simple and, therefore, the complete characterization of the roots of $P$ can be done applying Lemma~\ref{lemmaFactQuadratic} to the quadratic polynomial 
$(x-y_2)\ast(x-y_1)$.

We considered the application of the quaternionic Weierstrass  method to several examples of polynomials having double (isolated) or spherical roots and, in all the cases, we have observed the following:
when $y_1$ is a double isolated root ($y_1 \ne \overline{y}_2$), the behavior is analogous to the one observed in the classical case, i.e. the rate of convergence drops to one; on the other 
hand,  if $y_1$ is a spherical root, the iterative scheme produces two distinct roots $\zeta_1$ and $\zeta_2$ belonging to the congruence class $[y_1]$ and still shows a quadratic order of convergence. 

\section{Numerical examples}
In this section we present several examples illustrating the performance of the quaternionic Weierstrass method introduced in Section 3. 

All the numerical experiments here reported  were obtained by the use of the Mathematica add-on application \texttt{QuaternionAnalysis} \cite{MirandaFalcao2014} designed by two of the authors of this paper for symbolic manipulation of quaternion valued functions. 
A collection of new functions, including an implementation of the Weierstrass method described in this paper, has been recently developed in order to endow the aforementioned package with the ability to perform operations in  the non-commutative ring of polynomials $\HH [x]$.

\begin{example} \label{ex1}
Our first test example is a polynomial which fulfills the assumptions of Theorem~\ref{main}. 
In fact, it is easy to see that the polynomial
\begin{equation}\label{FactEx1}
	P(x)=(x+2\qi)\ast(x+1+\qk)\ast(x-2)\ast(x-1)\ast(x-2+\qj)\ast(x-1+\qi),
\end{equation}
has only simple roots, namely
\begin{align*}
	\zeta_1&=1-\qi,&  \zeta_2&=1,& \zeta_3&=-1-\tfrac{29}{39}\qi+\tfrac{14}{39}\qj-\tfrac{22}{39}\qk,\\  
	\zeta_4&=2,&\zeta_5&=-\tfrac{224}{113}\qi-\tfrac{30}{113}\qk,& \zeta_6&=2-\tfrac23\qi-\tfrac13\qj+\tfrac23\qk.
\end{align*}
Since, in this case, the polynomial roots $\zeta_i$ are known exactly, we replace the stopping criterion based on the incremental size of the iterations
by the following one: 
\begin{equation*}
\epsilon^{(k)}:=\max_i \{\epsilon_i^{(k)}\} < \varepsilon_1, \text{ with } \epsilon_i^{(k)}:=|\zeta_i^{(k)}-\zeta_{\pi_k(i)}|,
\end{equation*}
where $\pi_k$ is an appropriate permutation of $\{1,\dots,6\}$. 
Here, we considered $\varepsilon_1=\varepsilon_2=10^{-16}$ and chose initial approximations so that $\epsilon^{(0)}\le 0.5$.

The Weierstrass method applied to the extended form of $P$ produced, after 5 iterations, the following approximations to the factor terms (with 15 decimal places\footnote{The notation $(0)$ after the decimal point represents a sequence of 15 zeros.})
{\small
\begin{align*}
x_1^{(5)}&=1.(0)-1.(0)\qi\\
x_2^{(5)}&=1.(0)\\
x_3^{(5)}&=-1.(0)-0.545454545454545\qi-0.181818181818182\qj-0.818181818181818\qk\\
x_4^{(5)}&=2.(0)\\
x_5^{(5)}&=-1.587878787878788\qi-0.911515151515152\qj+0.804848484848485\qk\\
x_6^{(5)}&=2.(0)+0.133333333333333\qi+0.093333333333333\qj-0.986666666666667\qk
\end{align*}}
corresponding to the approximate roots
{\small
\begin{align*}
\zeta_1^{(5)}&=1.(0)-1.(0)\qi\\
\zeta_2^{(5)}&=1.(0)\\
\zeta_3^{(5)}&=-1.(0)-0.743589743589744\qi+0.358974358974359\qj-0.564102564102564\qk\\
\zeta_4^{(5)}&=2.(0)\\
\zeta_5^{(5)}&=-1.982300884955752\qi-0.265486725663717\qk\\
\zeta_6^{(5)}&=2.(0)-0.666666666666667\qi-0.333333333333333\qj+0.666666666666667\qk
\end{align*}}

It is interesting to observe that the approximations $x_i^{(5)}$ to the factor terms lead to a factorization of $P$ different from  \eqref{FactEx1}, but of course in line with Theorem~\ref{theoremResultsZeros}-(iii).

Table~\ref{table1} contains the relevant information concerning the errors in the successive approximations $\zeta_i^{(k)}$ ($k=0,\dots,5$, $i=1,\dots,6$) to the roots $\zeta_i$ of $P$. 
Estimates $\rho$ for the computational local order of convergence of the method, based on the use of (see e.g. \cite{SanchezNogueraGrauHerrero2012} for details).
\begin{equation*} 
\rho\approx \rho^{(k)}:=\frac{\log\epsilon^{(k)}}{\log\epsilon^{(k-1)}}
\end{equation*}
were also computed and are included in the last column of the table.

\begin{table}
\caption{Quaternionic Weierstrass method for Example~\ref{ex1}}\label{table1}
\begin{center}
\begin{tabular}{|l||l|l|l|l|l|l||l|}
\hline&&&&&&&\\[-2ex]	
$k$&\quad $\epsilon_1^{(k)}$ &\quad $\epsilon_2^{(k)}$&\quad $\epsilon_3^{(k)}$&\quad $\epsilon_4^{(k)}$&\  $\epsilon_5^{(k)}$&\quad $\epsilon_6^{(k)}$&\  $\rho^{(k)}$\\ 
\hline
\hline&&&&&&&\\[-2ex]
$0$&$2.7\,\text{e}{-1}$ &$6.0\,\text{e}{-2}$  &$4.5\,\text{e}{-1}$  &$2.0\,\text{e}{-2}$  &$8.0\,\text{e}{-2}$  &$3.3\,\text{e}{-1}$  & --\\
$1$&$9.3\,\text{e}{-2}$&$1.8\,\text{e}{-2}$ &$7.7\,\text{e}{-2}$ &$7.2\,\text{e}{-3}$ &$6.0\,\text{e}{-2}$ &$3.1\,\text{e}{-2}$ & 1.36\\
$2$&$7.9\,\text{e}{-3}$&$1.9\,\text{e}{-3}$ &$5.7\,\text{e}{-3}$ &$5.6\,\text{e}{-4}$ &$5.3\,\text{e}{-3}$ &$1.1\,\text{e}{-3}$ & 2.03\\
$3$&$6.0\,\text{e}{-5}$&$2.0\,\text{e}{-5}$ &$4.0\,\text{e}{-5}$ &$4.5\,\text{e}{-6}$ &$9.2\,\text{e}{-6}$ &$3.1\,\text{e}{-7}$ & 2.17\\
$4$&$2.4\,\text{e}{-9}$&$1.5\,\text{e}{-9}$ &$3.3\,\text{e}{-9}$ &$3.7\,\text{e}{-10}$&$2.0\,\text{e}{-9}$ &$2.2\,\text{e}{-13}$& 2.04\\
$5$&$1.5\,\text{e}{-17}$&$5.3\,\text{e}{-18}$&$1.6\,\text{e}{-17}$&$6.1\,\text{e}{-19}$&$7.5\,\text{e}{-18}$&$8.1\,\text{e}{-26}$& 2.06\\
\hline
\end{tabular}
\end{center}
\end{table}

In order to illustrate Remark~\ref{serial} we have also implemented the parallel version of Weierstrass method. In this case, using the same initial guesses, 9 iterations were required to achieve the same precision. 
The results presented in Table~\ref{table2} clearly indicate the deterioration of the speed of convergence of this version of the method.

\begin{table}
\caption{Parallel version of Weierstrass method for Example~\ref{ex1}}\label{table2} 
\begin{center}
\begin{tabular}{|l||l|l|l|l|l|l||l|}
   \hline&&&&&&&\\[-2ex]	
$k$&\quad $\epsilon_1^{(k)}$& $\ \epsilon_2^{(k)}$&\quad $\epsilon_3^{(k)}$&\quad $\epsilon_4^{(k)}$&\quad $\epsilon_5^{(k)}$&\quad $\epsilon_6^{(k)}$ &\ $\rho^{(k)}$\\ 
\hline
\hline&&&&&&&\\[-2ex]
$0$&$2.7\,\text{e}{-1}$ &$6.0\,\text{e}{-2}$  &$4.5\,\text{e}{-1}$  &$2.0\,\text{e}{-2}$  &$8.0\,\text{e}{-2}$  &$3.3\,\text{e}{-1}$  & --\\
$1$&$9.3\,\text{e}{-2}$&$2.3\,\text{e}{-2}$ &$8.6\,\text{e}{-2}$ &$5.9\,\text{e}{-3}$ &$5.3\,\text{e}{-3}$ &$9.8\,\text{e}{-2}$ & 3.26\\
$2$&$1.7\,\text{e}{-2}$&$4.4\,\text{e}{-3}$ &$2.8\,\text{e}{-2}$ &$2.0\,\text{e}{-3}$ &$2.0\,\text{e}{-2}$ &$3.3\,\text{e}{-2}$ & 1.47\\
$3$&$1.0\,\text{e}{-3}$&$4.3\,\text{e}{-4}$ &$3.1\,\text{e}{-3}$ &$3.1\,\text{e}{-4}$ &$2.9\,\text{e}{-3}$ &$4.2\,\text{e}{-3}$ & 1.60\\
$4$&$5.0\,\text{e}{-5}$&$1.9\,\text{e}{-5}$ &$1.5\,\text{e}{-4}$ &$1.2\,\text{e}{-5}$ &$2.4\,\text{e}{-4}$ &$3.0\,\text{e}{-4}$ & 1.48\\
$5$&$2.2\,\text{e}{-7}$ &$1.1\,\text{e}{-7}$ &$2.1\,\text{e}{-6}$ &$1.0\,\text{e}{-7}$ &$7.0\,\text{e}{-6}$ &$7.9\,\text{e}{-6}$ & 1.44\\
$6$&$4.5\,\text{e}{-10}$&$2.2\,\text{e}{-10}$&$5.6\,\text{e}{-8}$ &$2.5\,\text{e}{-10}$&$6.0\,\text{e}{-7}$ &$1.9\,\text{e}{-6}$ & 1.12\\
$7$&$2.5\,\text{e}{-13}$&$1.2\,\text{e}{-13}$&$9.4\,\text{e}{-11}$&$2.0\,\text{e}{-13}$&$1.2\,\text{e}{-8}$ &$8.6\,\text{e}{-9}$ & 1.39\\
$8$&$1.6\,\text{e}{-18}$&$1.6\,\text{e}{-18}$&$1.6\,\text{e}{-15}$&$3.7\,\text{e}{-18}$&$2.4\,\text{e}{-12}$&$3.0\,\text{e}{-12}$& 1.45\\
$9$&$2.0\,\text{e}{-25}$&$1.7\,\text{e}{-25}$&$7.9\,\text{e}{-21}$&$5.5\,\text{e}{-25}$&$1.5\,\text{e}{-17}$&$2.3\,\text{e}{-17}$& 1.44\\
\hline
\end{tabular}
\end{center}
\end{table}
\end{example}

Our next examples concern situations where the polynomials under consideration have zeros which are not simple.

\begin{example} \label{ex2}
The polynomial 
\begin{equation*} 
P(x)=x^4+(-1+\qi)x^3+(2-\qi+\qj+\qk)x^2+(-1+\qi)x+1-\qi+\qj+\qk,
\end{equation*}
has, apart from the isolated zeros  $-\qi+\qk$ and  $1-\qk$, a 
whole sphere  of zeros,  $[\,\qi\,]$.
In this case, since all the spherical roots have the same  real part and modulus,  we replaced the stopping criterion used in the previous example by the following one:
\begin{equation*} 
\epsilon^{(k)}=\max\{\epsilon_R^{(k)},\epsilon_N^{(k)}\}< 10^{-16},
\end{equation*}
where
\begin{equation*}
\epsilon_R^{(k)}  := \max_i\{\re(\zeta_i^{(k)})-\re(\zeta_{\pi_k(i)})\}\quad \text{and}\quad
\epsilon_N^{(k)} := \max_i\{\big||\zeta_i^{(k)}|-|\zeta_{\pi_k(i)}|\big|\}.
\end{equation*}
	
Starting with an initial guess chosen so that $\epsilon^{(0)}\le 0.15$, we obtained, after 5 iterations, the following approximations:
{\small
\begin{align*}
\zeta_1^{(5)}&=0.099934477851162\qi -0.917198737816235\qj -0.385693629043728\qk\\
\zeta_2^{(5)}&= -0.799427021998164\qi -0.519295977566198\qj -0.302073044449043\qk\\
\zeta_3^{(5)}&= 1.(0)-1.(0)\qj \\
\zeta_4^{(5)}&= -1.(0)\qi+1.(0)\qk
\end{align*}}

The spherical root can be identified at once by observing that, up to the required precision, we have $[\zeta_1^{(5)}]=[\zeta_2^{(5)}]$, since $\re \zeta_1^{(5)}=\re \zeta_2^{(5)}=0$ and $|\zeta_1^{(5)}|=|\zeta_2^{(5)}|=1$.
 
The numerical details related to this example are displayed in Table~\ref{table3}. Here the numerical computations have been carried  out with the precision increased to 512 significant digits.

\begin{table} 
\caption{Weierstrass method for spherical roots - Example~\ref{ex2}}\label{table3}
\begin{center}
\begin{tabular}{|l||l|l|l|l||l|}
\hline&&&&&\\[-2ex]	
$k$&\quad $\epsilon_1^{(k)}$&\quad $\epsilon_2^{(k)}$&\quad $\epsilon_3^{(k)}$&\quad $\epsilon_4^{(k)}$&\ $\rho^{(k)}$\\ 
\hline
\hline&&&&&\\[-2ex]
$0$&$1.3\,\text{e}{-2}$ &$7.1\,\text{e}{-2}$ &${7.6\,\text{e}{-2}}$ &${1.3\,\text{e}{-1}}$ & --\\
$1$&$6.3\,\text{e}{-3}$ &$4.5\,\text{e}{-2}$ &${6.1\,\text{e}{-3}}$ &${1.1\,\text{e}{-2}}$ & 1.52\\
$2$&$1.2\,\text{e}{-4}$ &$9.8\,\text{e}{-4}$ &${1.4\,\text{e}{-3}}$ &${9.7\,\text{e}{-5}}$ & 2.11\\
$3$&$9.6\,\text{e}{-8}$ &$1.1\,\text{e}{-6}$ &${2.6\,\text{e}{-6}}$ &${1.8\,\text{e}{-8}}$ & 1.96\\
$4$&$9.2\,\text{e}{-12}$ &$6.1\,\text{e}{-11}$ &${1.6\,\text{e}{-11}}$ &${1.0\,\text{e}{-15}}$ & 1.83\\
$5$&$1.4\,\text{e}{-22}$&$9.7\,\text{e}{-22}$&${4.9\,\text{e}{-21}}$&${8.0\,\text{e}{-31}}$& 1.99\\
\hline
\end{tabular}
\end{center}
\end{table}

As we can observe from Table~\ref{table3}, the quaternionic Weierstrass method works, produces {\sl all} the roots simultaneously with machine precision and exhibits quadratic order of convergence.\
As expected, for the case of the spherical root, we obtain convergence to two distinct members of the sphere of zeros.
\end{example}
\begin{example}\label{ex3}
In our last example we address the problem of using Weierstrass method in cases where the polynomial under consideration has multiple (isolated) roots. The polynomials
\begin{equation*}
P(x)=(x-\qi)\ast(x+1+\qk)\ast(x+1+\qk) \text{ and } Q(x)=(x-\qi)\ast(x+1-\qi)\ast(x+1+\qk)
\end{equation*}
have one non-real root with multiplicity one and $-1-\qk$ as a double root. The approximations to the roots of $P$ obtained by the use of the quaternionic Weierstrass method are
\begin{align*}
\zeta_1=&-1.(0) -1.(0)\qk \\
\zeta_2=& -0.230769230769231\qi -0.307692307692308\qj-0.923076923076923\qk\\
\zeta_3=&-1.(0)  -1.(0)\qk
\end{align*}
while, for the roots of $Q$, we obtained
{\small
\begin{align*}
\zeta_1=& 0.333333333333333\qi -0.666666666666667\qj -0.666666666666667\qk\\
\zeta_2=&-1.(0) -1.(0)\qk \\
\zeta_3=&-1.(0)  -1.(0)\qk
\end{align*}}
 As we can observe from Table~\ref{table4}, the behavior of the quaternionic Weierstrass method is very similar to that one observed for the classical complex case, where the rate of convergence is linear. This table shows $\epsilon^{(k)}$ for the last 9 iterations of the  method together with $\rho^{(k)}$ for both polynomials.

\begin{table} 
\caption{Weierstrass method for double roots - Example~\ref{ex3}}\label{table4}
\begin{center}
\begin{tabular}{|l|l||l|l|}
\hline
\multicolumn{2}{|c||}{$P$}&\multicolumn{2}{|c|}{$Q$}\\
\hline
&&&\\[-2ex]
\quad $\epsilon^{(k)}$ &\  $\rho^{(k)}$ &\quad  $\epsilon^{(k)}$ &\  $\rho^{(k)}$\\
\hline
&&&\\[-2ex]
$6.8\,\text{e}{-10}$& 1.05 &$1.9\,\text{e}{-13}$ & 1.03 \\ 
$2.6\,\text{e}{-10}$& 1.04 &$7.3\,\text{e}{-14}$ & 1.03 \\ 
$1.0\,\text{e}{-10}$& 1.04 &$2.8\,\text{e}{-14}$ & 1.03 \\ 
$3.9\,\text{e}{-11}$& 1.04 &$1.1\,\text{e}{-14}$ & 1.03 \\ 
$1.5\,\text{e}{-11}$& 1.04 &$4.1\,\text{e}{-15}$ & 1.03 \\ 
$5.9\,\text{e}{-12}$& 1.04 &$1.6\,\text{e}{-15}$ & 1.03 \\ 
$2.3\,\text{e}{-12}$& 1.04 &$6.1\,\text{e}{-16}$ & 1.03 \\ 
$8.8\,\text{e}{-13}$& 1.04 &$2.4\,\text{e}{-16}$ & 1.03 \\ 
$3.4\,\text{e}{-13}$& 1.03 &$9.1\,\text{e}{-17}$ & 1.03 \\ 

\hline 
\end{tabular} 
\end{center}
\end{table}

\end{example}
 
\section{Final Remarks}

In this paper we proposed a  generalization to the  quaternionic context of the well-known Weierstrass method for approximating {\sl all} zeros of a polynomial simultaneously. 
Due to the structure of the zero-set of a quaternionic polynomial, the claim that the method we have proposed produces {\sl all} the zeros simultaneously, requires an additional explanation.
Assuming the convergence of the method to  the roots $\zeta_1,\dots,\zeta_n$ of a polynomial $P$ of degree $n$, it is easy to identify ${\mathbf{Z}}_P$, once we test if each element of $\{\zeta_1,\dots,\zeta_n\}$ is an isolated or a spherical zero of $P$ (cf.~Theorem~\ref{teorSphericalRoots}).

The quaternionic Weierstrass algorithm is entirely based on quaternionic arithmetic and shows fast convergence for simple and spherical roots.
We proved the quadratic convergence of the sequential iterative scheme, under the assumptions that all the roots of the polynomial are simple, and presented  numerical examples supporting this fact.
In \cite{Falcao2014}, it was proved that the same rate of convergence can be achieved by quaternion versions of Newton's method for the so-called radially holomorphic functions \cite[p.~234]{GurlebeckHabethaSprossig2008}.
None of the polynomials presented in this section are radially holomorphic or are in the less restrictive conditions of \cite[Theorem~4]{Falcao2014}.
As far as we are aware, the method proposed in this paper is the first numerical method entirely based on quaternionic arithmetic for which we can observe theoretical and experimental results for general unilateral quaternion polynomials.

Several authors, namely Petkovic and collaborators (see e.g. \cite{Petkovic2008}), have described conditions for the safe convergence of the classical method depending only on the initial approximations.  This is a very interesting question that we intend to address in the near future, in the quaternionic case.

One can find in the literature several modifications to the classical Weierstrass method which improve the speed of convergence to multiple roots (see e.g. \cite{Fraigniaud1991}). It is also in our plans of research to consider adaptations of such strategies.

\section*{Acknowledgments}
Research at CMAT was financed by Portuguese Funds through FCT - Funda\c{c}\~{a}o para a Ci\^{e}ncia e a Tecnologia - within the Project UID/MAT/00013/2013. Research at NIPE has been carried out within the funding with
COMPETE reference number POCI-01-0145-FEDER-006683, with the FCT/MEC's financial support
through national funding and by the ERDF through the Operational Programme
on ``Competitiveness and Internationalization -- COMPETE 2020" under the
PT2020 Partnership Agreement.


\end{document}